\documentclass[12pt,a4paper]{article}
\usepackage[latin1]{inputenc}
\usepackage{amsfonts,amsmath,amssymb,amsthm}
\usepackage{graphicx}
\usepackage{hyperref}

\renewcommand{\epsilon}{\varepsilon}

\newcommand{\Ex}{\mathbb{E}}
\newcommand{\Prob}{\mathbb{P}}

\newtheorem{theorem}{Theorem}

\newtheorem{lemma}{Lemma}

\begin{document}

\title{Reduced branching processes with very heavy tails}
\author{Andreas N. Lagerås\thanks{Centre for Theoretical Biology, Göteborg University.} \thanks{Address: Mathematical Sciences, Chalmers University of Technology, 412 96 Göteborg, Sweden.}
\and
Serik Sagitov\footnotemark[\value{footnote}]}
\maketitle
\begin{abstract}
The reduced Markov branching process is a stochastic model for the genealogy of an unstructured biological population. Its limit behavior in the critical case is well studied for the Zolotarev-Slack regularity parameter $\alpha\in(0,1]$. We turn to the case of very heavy tailed reproduction distribution $\alpha=0$ assuming Zubkov's regularity condition with parameter $\beta\in(0,\infty)$. Our main result gives a new asymptotic pattern for the reduced branching process conditioned on non-extinction during a long time interval. \\

\end{abstract}

\section{Introduction}

A single type branching process describes a population of particles with independent and identical reproduction laws. In the Markov branching process with continuous time each particle lives an exponential time with mean one and at death splits into a random number of daughter particles. If we assume that the branching system starts at time zero from a single particle with $\nu$ daughters, then the whole process is defined by the distribution of the random variable $\nu$. In the critical case when the average number of daughters is exactly one $\Ex[\nu]=1$, the limit behavior of the branching process is studied under the following Zolotarev-Slack regularity condition (cf.\  \cite{Zo} and  \cite{Sl}). The generating function $f(s)=\Ex[s^\nu]$ is assumed to satisfy
\begin{equation}\label{alpha}
    f(s)=s+(1-s)^{1+\alpha}L \left({1\over1-s}\right),\ 0\le\alpha\le1,
\end{equation}  
where $L$ is slowly varying at infinity. This condition is valid with $\alpha=1$ if the variance of $\nu$ is finite, while the case $0<\alpha<1$ is usually referred to as the infinite variance case and is well studied in the literature. 

We turn to the less studied case $\alpha=0$, focussing on a special class of slowly varying functions (which was initially introduced by Zubkov \cite{Zu})
\begin{equation}\label{beta}
L(x)\sim(\ln x)^{-\beta}L_1(\ln x),\ \beta>0,\ x\to\infty,
\end{equation}  
where $L_1$ is another slowly varying function. Given \eqref{alpha} and $\alpha=0$, Zubkov's regularity condition \eqref{beta} is equivalent to the next requirement on the tail distribution function (see Lemma \ref{l1})
\begin{equation}\label{beta1}
    \Prob(\nu>k)\sim\beta k^{-1}(\ln k)^{-1-\beta}L_1(\ln k),\ k\to\infty.
\end{equation}
Notice that in this case $\Ex[\nu(\ln_+\nu)^{\beta-\epsilon}]<\infty$ and $\Ex[\nu(\ln_+\nu)^{\beta+\epsilon}]=\infty$ for all $\epsilon>0$. This is a consequence of \cite[Thm 8.1.8]{BGT}.

The main characteristic of the branching process is the number of particles $Z(t)$ alive at time $t$. The key issues of the asymptotics of the non-extinction probability $Q(t)=\Prob(Z(t)>0)$  as $t\to\infty$ and the limit behavior of $Z(t)$ conditioned on non-extinction in the case $\alpha=0$ were recently addressed by Nagaev and Wachtel \cite{NW}. They consider the discrete time version of the Markov branching process and obtain general limit results without the extra assumption \eqref{beta}. By repeating Zubkov's arguments \cite[p.\ 607]{Zu}, our results could be carried over to discrete time with no change. In Section \ref{s4} we give a direct proof (which is more straightforward than the counterpart of the Nagaev-Wachtel proof) of the following result with the extra condition. 

\begin{theorem}\label{th}
If \eqref{alpha} holds with $\alpha=0$ and $L$ satisfies \eqref{beta}, then

\begin{equation}\label{Q}
    Q(t)=\exp\left(-t^{1/(1+\beta)} L_q(t)\right),
\end{equation} 
where $L_q(t)$ is such a slowly varying function as $t\to\infty$ that
\[L_q^{1+\beta}(t)\sim(1+\beta)L_1(t^{1/(1+\beta)} L_q(t)).\]
Furthermore, there exists a regularly varying function (see \eqref{cc})
\begin{equation}\label{c}
    c(t)=t^{\beta(1+\beta)^{-2}}L_c(t)
\end{equation}
such that for all $x\ge0$
\begin{equation}\label{lt}
P(Z(t)\le e^{xc(t)}|Z(t)>0)\to1-e^{-x^{\beta+1}},\ t\to\infty.
\end{equation}
\end{theorem}

There is a striking feature in the asymptotics of $Z(t)$ which was pointed out to us by V.\ Wachtel. Notice that the regular variation index $\beta(1+\beta)^{-2}$ of the scaling function $c(t)$ increases as $\beta$ goes from infinity down to $\beta=1$. As the reproduction tail becomes heavier this is what we expect to happen, namely to have larger  asymptotic value for the population size $Z(t)$ at survival. What is puzzling about \eqref{lt} however, is that as $\beta$ falls below the threshold value 1, the corresponding scaling function $c(t)$ attributes smaller size for the surviving population despite the fact that the reproduction tail becomes even heavier. 

Some light on this phenomenon is shed by the following seminal results by Zubkov \cite{Zu}. If $\tau(t)$ is the time to the most common ancestor for all particles alive at time $t$, then under condition \eqref{alpha} with $0<\alpha\le1$ 
\begin{equation}\label{zu0}
    \Prob(\tau(t)\le tx|Z(t)>0)\to x,\ t\to\infty,
\end{equation}
while under the conditions of Theorem \ref{th} 
\begin{equation}\label{zu}
    \Prob(\tau(t)\le tx|Z(t)>0)\to x^{\beta/(1+\beta)},\ t\to\infty.
\end{equation}
The latter means that the ratio $\tau(t)/t$ is asymptotically distributed over $[0,1]$ with the density function
\begin{equation}\label{fi}
  \phi_\beta(x)=\beta(1+\beta)^{-1}x^{-1/(1+\beta)}.
\end{equation} 
That is, if $\beta$ is changed towards smaller values, then the time to the most recent common ancestor will become shorter and one can expect an eventual drop in the size of the surviving population as $\beta$ goes below a certain threshold value. Why the threshold value should be $\beta=1$ is an interesting \emph{open problem}. 

Section \ref{s3} presents the so-called reduced branching process describing the genealogy of the particles alive at time $t$. In this section we recall the known limit processes for the reduced branching processes in the cases $0<\alpha\le1$ obtained in \cite{FSS} and \cite{Ya}. Then we state the main result of this paper, Theorem \ref{the}, giving a new limit structure as $t\to\infty$ for the reduced branching process in the case $\alpha=0$ and \eqref{beta}. Our Theorem \ref{the} is an extension of \eqref{zu} in the same manner as the results by \cite{FSS} and \cite{Ya} are extensions of  \eqref{zu0}. It is worth mentioning that in Theorem \ref{the} we do not loose much generality by assuming \eqref{beta} in the case $\alpha=0$ since by Zubkov's proof it actually follows that non-degenerate limit distributions of $\tau(t)/t$ must be either of the form \eqref{zu0} or \eqref{zu} for critical branching processes. In Section \ref{s2} we establish some preliminary results, and in Section \ref{s4} we prove Theorems \ref{th} and \ref{the}.

\section{Limit theorem for the reduced branching process}\label{s3}
Let $Z(u,t)$ stand for the number of particles at time $u$ which stay alive or have descendants at a later time $t\ge u$. For a given time horizon $t$ the process $\{Z(u,t), 0\le u\le t\}$ is called the reduced branching process. The term \emph{reduced} reflects the fact that we count only those branches in the full genealogical tree that reach to the time of observation $t$. Clearly,
\begin{equation}\label{sam}
    \Prob(Z(u,t)=1)=\Prob(\tau(t)\le t-u),
\end{equation}
where $\tau(t)$ is the time to the most recent common ancestor. If \eqref{alpha} holds with $0<\alpha\le1$, then the limit theorem \eqref{zu0} for the time to the most recent ancestor is extended to the following limit theorem for the reduced branching process by \cite{FSS} and \cite{Ya}
 \begin{equation}\label{lpa}
    (Z(tx,t)|Z(t)>0)\overset{d}{\to} Z_{\alpha}(-\log(1-x)),
\end{equation}
where the convergence of the processes over the time interval $0\le x<1$ holds in the Skorohod sense.

Here the limit is a time transformed Markov branching process $Z_{\alpha}(t)$, whose generating functions can be written down explicitly. The particles of this process live exponential times with mean one and at the moment of death produce offspring with the generating function
$$
f_{\alpha}(s) = s + \frac{1}{\alpha}(1-s)((1-s)^{\alpha}-1).
$$
This process is supercritical with the offspring number $\nu_\alpha$ having mean $f'_{\alpha}(1)=1+1/\alpha$ and a distribution over $k=2,3,\dots$ with probabilities
\begin{equation}\label{na}
  \Prob(\nu_\alpha=k) = \left\{\begin{aligned}
&\mathbf{1}_{\{k=2\}}, &&\text{if }\alpha = 1, \\
&\frac{(1+\alpha)\Gamma(k-1-\alpha)}{\Gamma(1-\alpha)\Gamma(k+1)}, &&\text{if }0<\alpha <1,
\end{aligned}\right.
\end{equation}
where $\Gamma$ is the Gamma function. It is easily checked that the following generating function
\[F_{\alpha}(s,t) = 1-(1-e^{-t}+e^{-t}(1-s)^{-\alpha})^{-1/\alpha}\]
solves the forward Kolmogorov equation 
$${\partial F_{\alpha}(s,t)\over \partial t}=(f_{\alpha}(s)-s){\partial F_{\alpha}(s,t)\over\partial s}$$
thus providing the generating function of $Z_\alpha(t)$. Note that $Z_1(t)$ is the well-known Yule process of binary splitting with
$$
F_1(s,t) = \frac{s e^{-t}}{1-(1-e^{-t})s},
$$
implying that the distribution of $Z_1(t)$ is shifted geometric for all $t$.

To state our main result we introduce a Markov branching process $Z_0(t)$ with infinite mean for the offspring number. Here the reproduction law is given by 
\begin{equation}\label{n0}
    \Prob(\nu_0=k) =\frac{1}{k(k-1)},\ k=2,3,\ldots
\end{equation}
corresponding to the generating function
$$
f_0(s) = s + (1-s)\log(1-s).
$$
Thus the part of the formula \eqref{na} given for $0<\alpha<1$ is also valid for $\alpha=0$. The generating function of $Z_0(t)$ is
\begin{equation}\label{limit_dist}
F_{0}(s,t) = 1-(1-s)^{e^{-t}}
\end{equation}
which can be used to compute 
\begin{equation}\label{1}
    \Prob(Z_0(t)=1)=e^{-t}.
\end{equation}
\begin{theorem}\label{the}
Under the conditions of Theorem \ref{th} the weak convergence of the processes over the time interval $0\le x<1$ 
 $$(Z(tx,t)|Z(t)>0)\overset{d}{\to} Z_{0}\left(-{\beta\over1+\beta}\log(1-x)\right)$$
holds in the Skorohod sense as $t\to\infty$.
\end{theorem}

From this result it easy to recover \eqref{zu} using \eqref{sam} and \eqref{1}. The limit process $Z_{0}\left(-\beta(1+\beta)^{-1}\log(1-x)\right)$ gives the following algorithm defining the genealogical tree:
\begin{itemize}
    \item start with a single particle at time 0 which lives a random time $1-\tau_0$, where $\tau_0$ has density function  \eqref{fi},
    \item at the time $1-\tau_0$ split the initial particle into a random number of daughter particles according to the distribution \eqref{n0},
    \item given $\tau_0$, let each daughter particle, independently of other particles, mimic the life of its mother, namely let it live a time $(1-\tau_1)\tau_0$, where $\tau_1$ has density \eqref{fi} and then split it using \eqref{n0},
    \item given $\tau_0$ and  $\tau_1$, let a granddaughter particle live a time $(1-\tau_2)\tau_1\tau_0$, where $\tau_2$ has density  \eqref{fi} and then split it using \eqref{n0}, and so on.
\end{itemize}
This should be compared with a similar algorithm describing the limit process in \eqref{lpa} for $0<\alpha\le1$, where the density \eqref{fi} is replaced by the uniform density over $[0,1]$, and the offspring number distribution \eqref{n0} is replaced by \eqref{na}. 

Figure \ref{figuren} clearly indicates an interesting transformation of the limit law for the genealogical tree as the parameter $\alpha$ decreases from 1 to 0 and then at $\alpha=0$ the new parameter $\beta$ goes from $\infty$ down to 0. The $\alpha$-model has common branch length distribution and the value of parameter $\alpha$ determines the reproduction law, which smoothly changes from the deterministic splitting at $\alpha=1$ via \eqref{na} to the distribution \eqref{n0} with infinite mean at $\alpha=0$. It is easy to verify the stochastic domination property: if $0\leq \alpha_1 \leq \alpha_2 \leq 1$, then $\Prob(\nu_{\alpha_2}>k)\leq \Prob(\nu_{\alpha_1}>k)$ for all $k$. This property is nicely illustrated by simulations on the left part of Figure \ref{figuren}: the lower is the value of the parameter $\alpha$ the faster is the growth of the genealogical tree.

At $\alpha=0$ when the new parameter $\beta$ takes over the control, the dynamics of tree behavior drastically changes.  As $\beta$ goes from larger to smaller values it is the branch length (and not the reproduction) that undergoes transformation. Since the reproduction law \eqref{n0} is common for all $\beta\in(0,\infty)$, we observe an opposite development of the tree growth: the closer $\beta$ is to zero (and therefore the heavier is the tail of the  original reproduction law), the closer the splitting times are located  to the observation time. Here we have another example of the phenomenon mentioned earlier concerning the critical value $\beta=1$. We observe a growth pattern (this time in terms of genealogical trees) which reaches its top and then is followed by a monotone decline. 

\textbf{Remarks} This development in the tree growth depending on decreasing values of $\beta$ indicates that at the region $\beta=0$ the linear time scaling for the reduced process should be replaced by a non-linear one in agreement with \cite[Thm 4(b)]{Zu}. 

It was pointed out in \cite{Sa} that the limit reproduction law for the critical reduced branching process with $\alpha\in(0,1]$ is related to the merging law of the $\Lambda$-coalescent with $\Lambda(dx) = (1-\alpha)x^{-\alpha}dx$. Namely, if $Y_n$ is the size of the next merger given there currently are $n$ branches, then
$$\Prob(Y_n = k )  = \Prob(\nu_{\alpha} = k| \nu_{\alpha} \leq n).$$
Theorem \ref{the} shows that there is a similar link between the reduced processes with $\alpha = 0$ and the $\Lambda$-coalescent with uniform $\Lambda$, i.e.\ the Bolthausen-Sznitman coalescent (see \cite{Pi}).

\begin{figure}
\centering
\includegraphics[width=6.5cm]{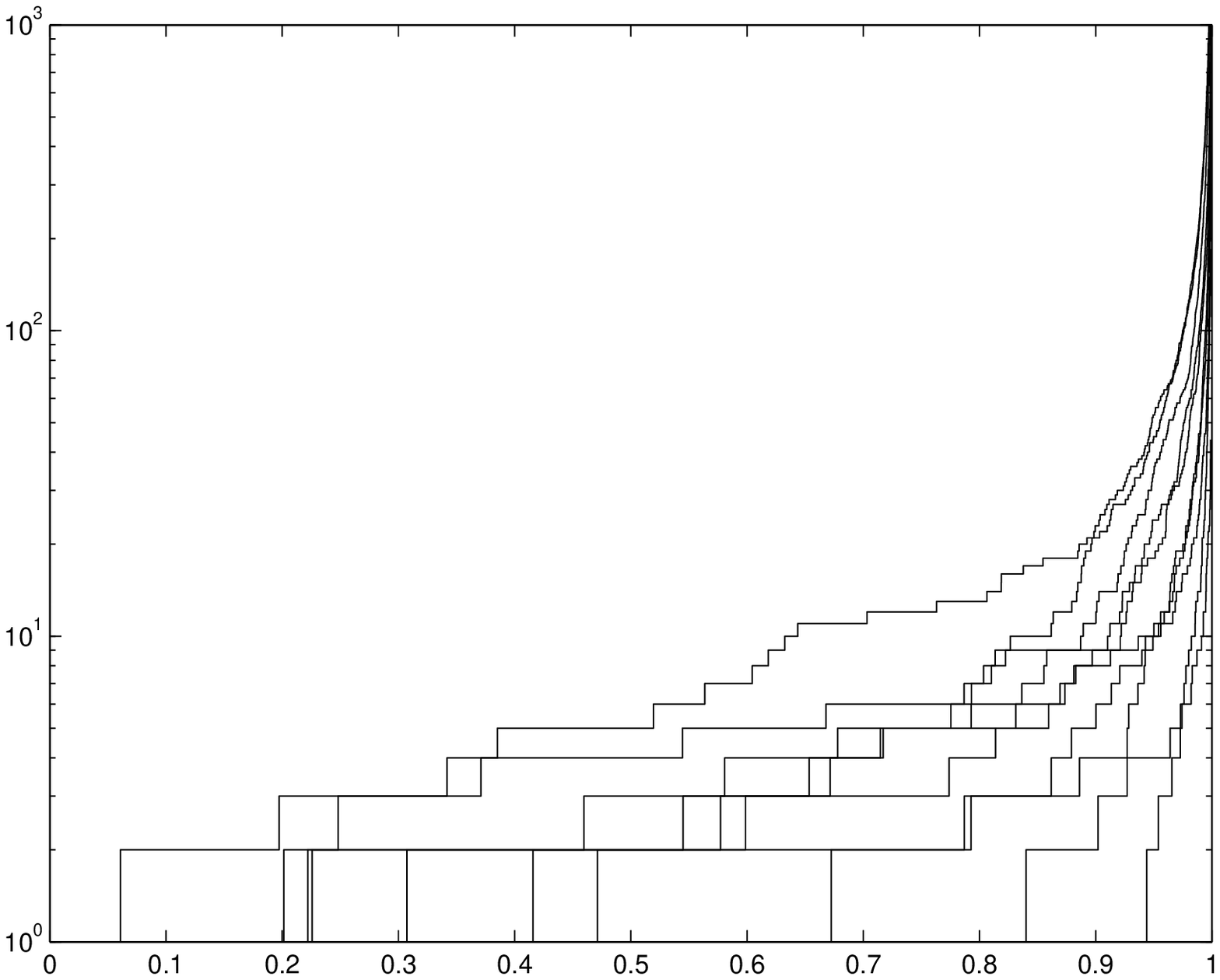}\includegraphics[width=6.5cm]{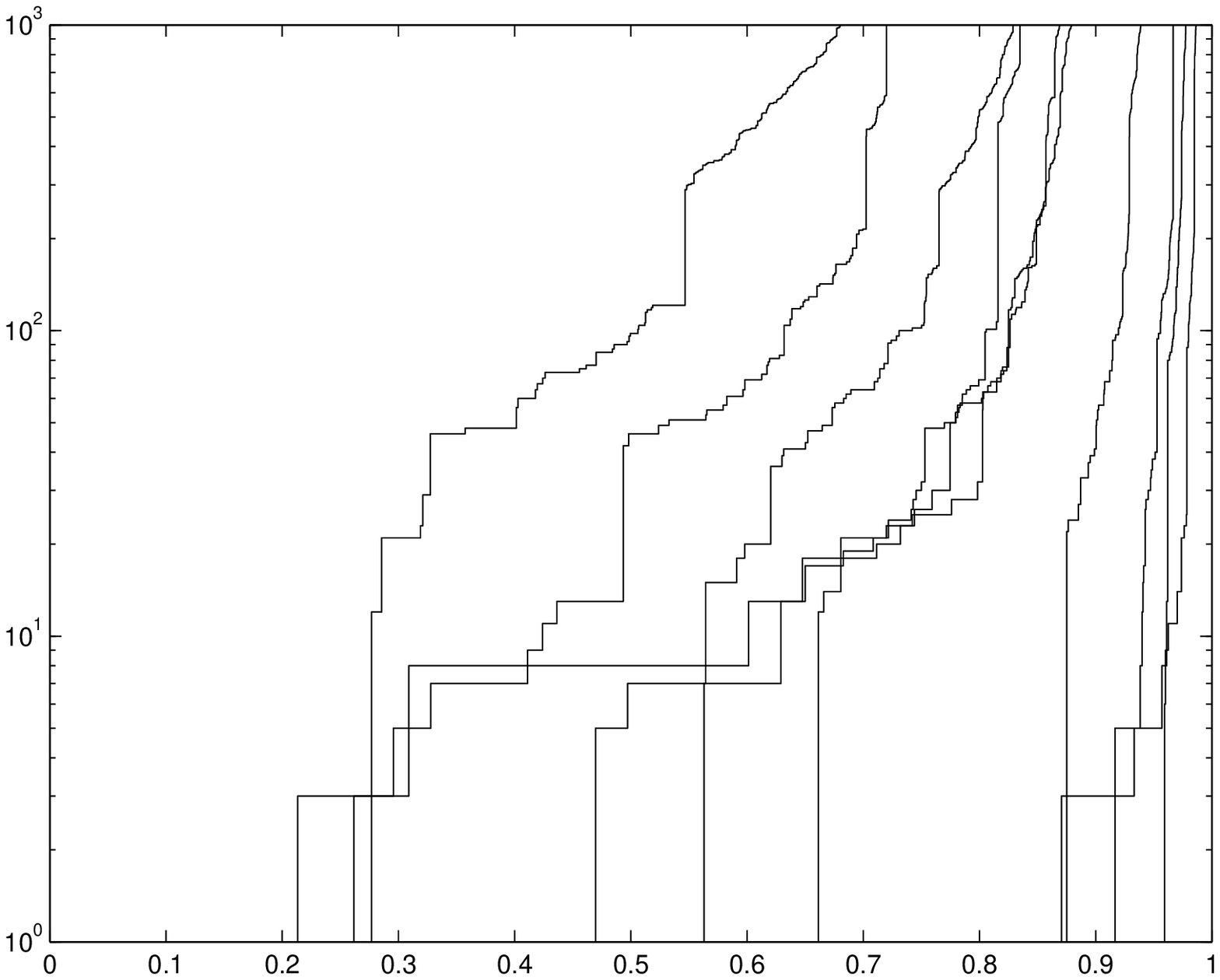}
\includegraphics[width=6.5cm]{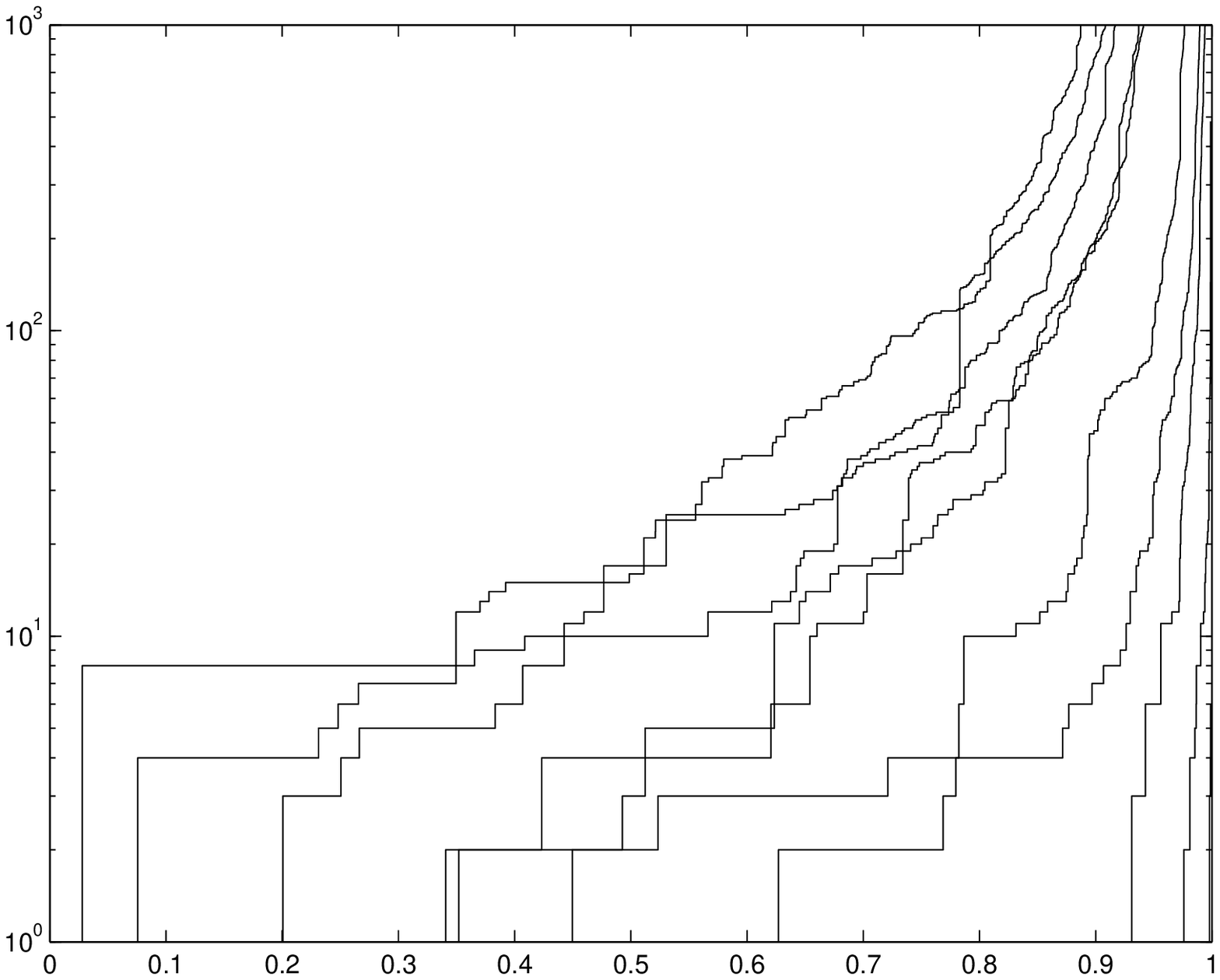}\includegraphics[width=6.5cm]{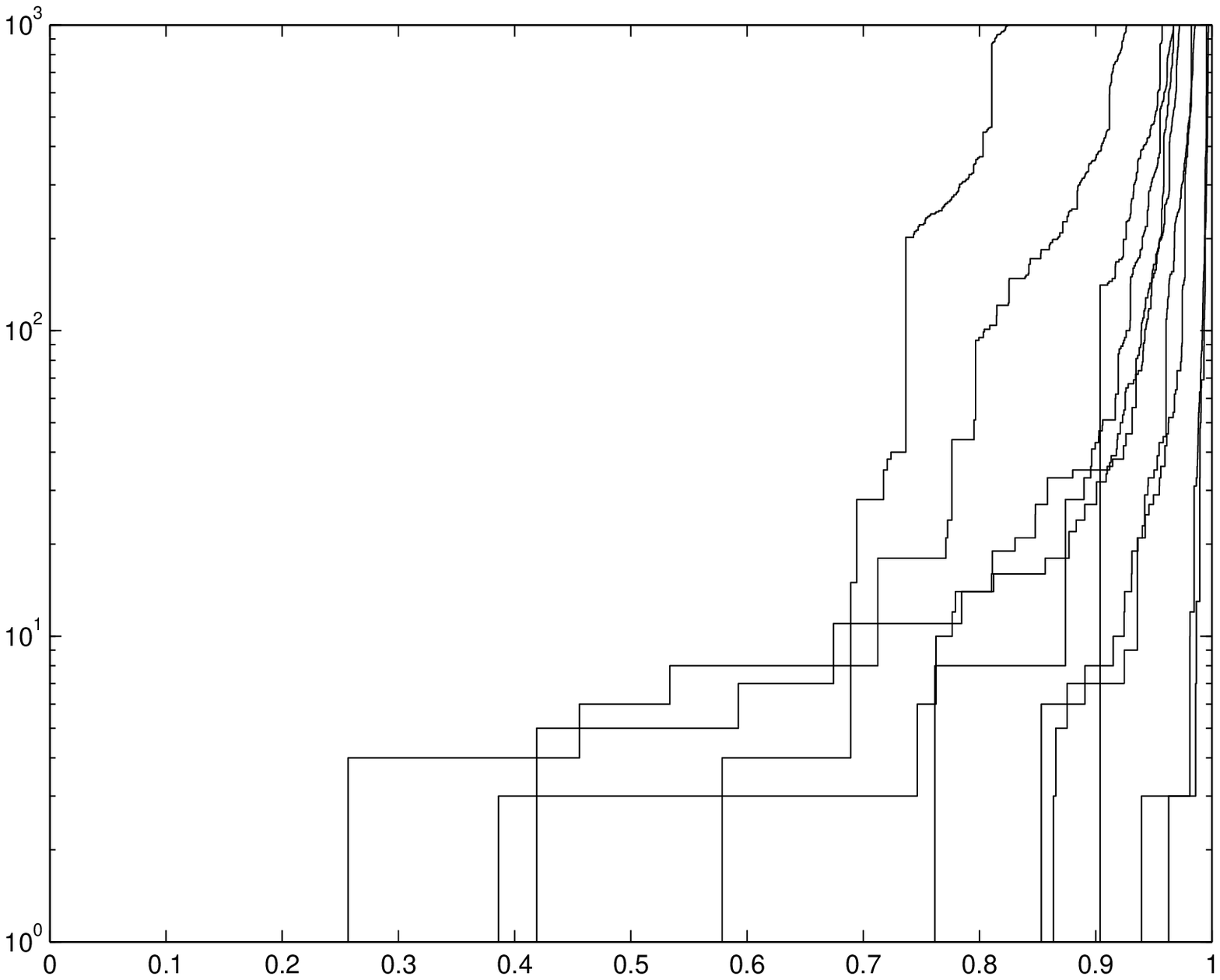}
\includegraphics[width=6.5cm]{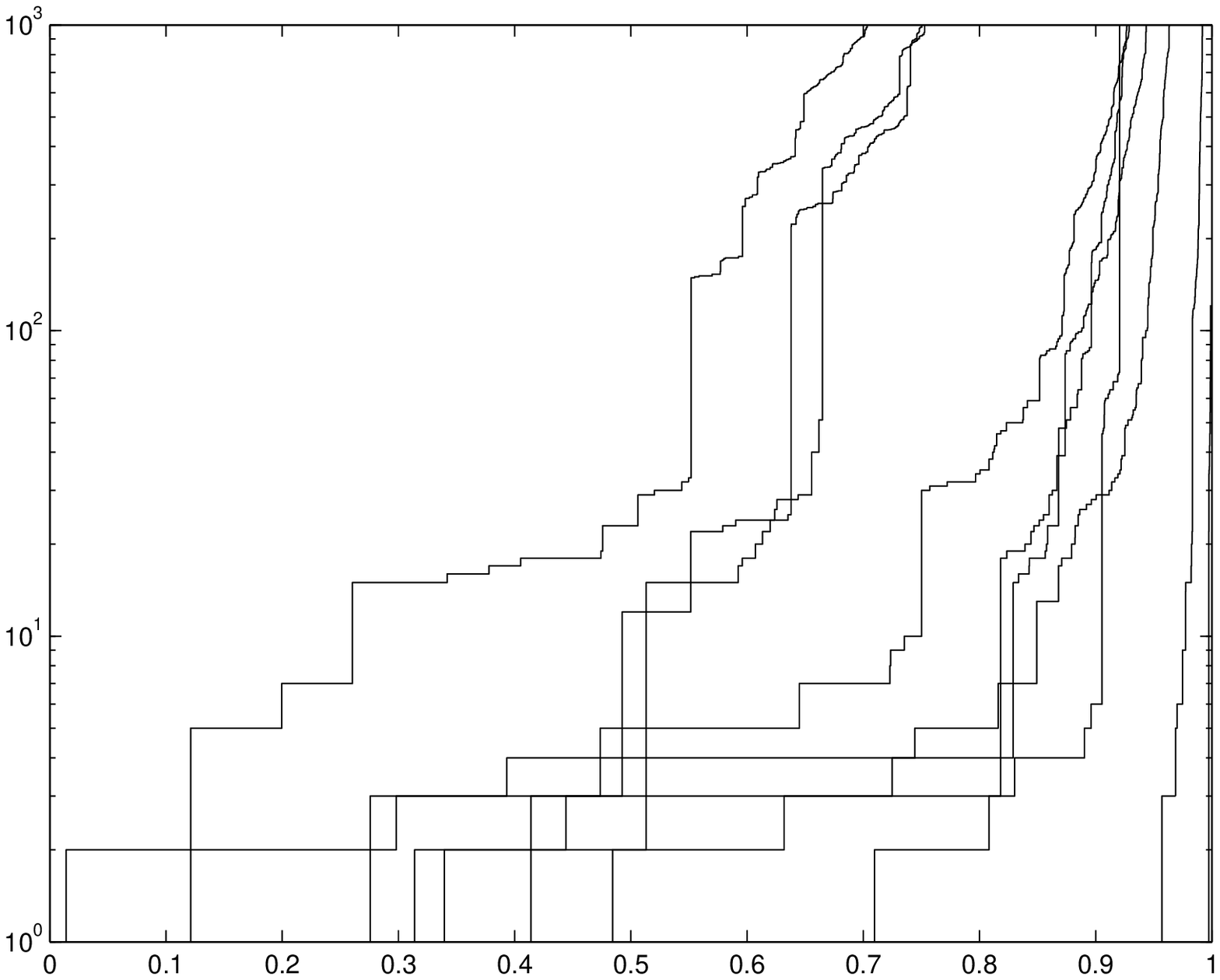}\includegraphics[width=6.5cm]{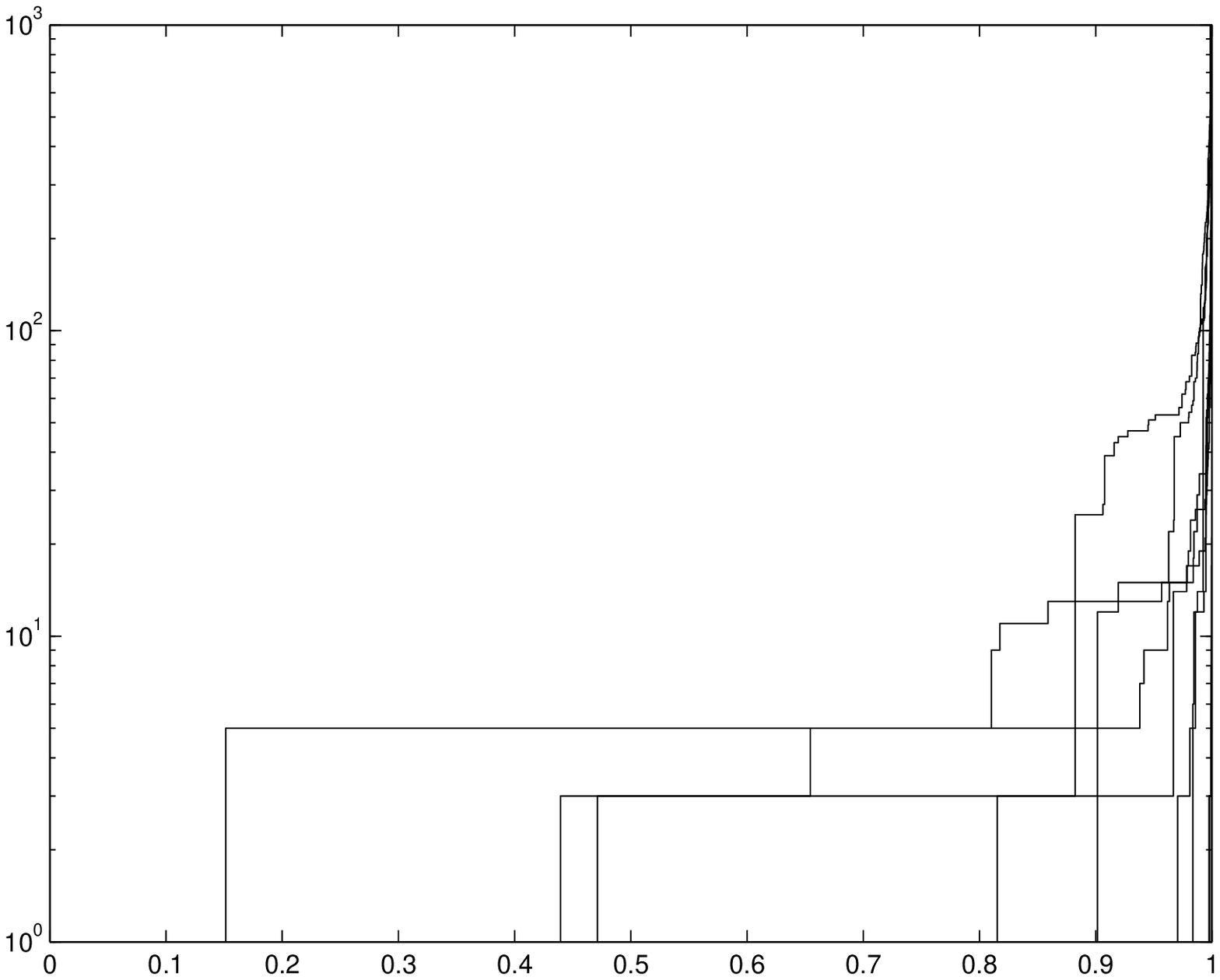}
\caption{Six blocks with ten realizations each of the limit process for different parameter values. Three blocks on the left correspond to the values: top $\alpha = 1$, middle $\alpha=0.3$, bottom $\alpha=0.1$.  Three blocks on the right correspond to the values: top $\beta = 5$, middle $\beta=1$, bottom $\beta=0.2$. For $\beta=0.2$, some trajectories are too close to 1 to be visible. 
Time to the most recent common ancestor is the horizontal distance from the right end of the time interval $[0,1)$ to the point where the trajectory leaves the state $1 = 10^0$.}\label{figuren}
\end{figure}

\section{Preliminary results}\label{s2}
If $f(s)>s$ for $0\le s<1$, the function 
\[\pi(s)=\int_0^s\frac{dv}{f(v)-v}\]
is obviously monotone and consequently $\rho(x) = \pi(1-e^{-x})$ as well. Let \eqref{alpha} hold with $\alpha=0$ and put $g(x)=L(e^x)$, then condition \eqref{beta} is equivalent to 
\begin{equation}\label{g}
    g(x)\sim x^{-\beta}L_1(x),\ x\to\infty.
\end{equation}
Note that
\begin{align*}
\pi(s)&=\int_{1-s}^1\frac{dt}{f(1-t)-1+t}\\
&=\int_{1-s}^1\frac{1}{L(1/t)}\frac{dt}{t}=\int_0^{-\ln(1-s)}\frac{dw}{L(e^{w})}=\int_0^{-\ln(1-s)}\frac{dw}{g(w)},
\end{align*}
therefore
\begin{equation}\label{rg}
   \rho(x)=\int_0^xdw/g(w),
\end{equation}
and it follows from \eqref{g} and \cite[Thm 1.5.8]{BGT} that
\begin{equation}\label{r}
    \rho(x)\sim {x^{\beta+1}\over(\beta+1)L_1(x)},\ x\to\infty.
\end{equation}

\begin{lemma} \label{l1}  Put $q(t) = -\ln Q(t)$, then 
\begin{eqnarray}
  \rho(q(t)) &=&t, \label{rq}\\
  q'(t) &=& g(q(t)).\label{q'} 
\end{eqnarray}  
\end{lemma}
\begin{proof} The generating function $F(s,t) =\Ex[s^{Z(t)}]$ of the Markov branching process satisfies the backward Kolmogorov equation
\[\frac{\partial F(s,t)}{\partial t}=f(F(s,t))-F(s,t)\]
 with the boundary condition $F(s,0)=s$. It follows that
\begin{equation}\label{raz}
    \pi(F(s,t)) = \pi(s) + t.
\end{equation}
Putting $s=0$ yields the asserted equality \eqref{rq}
 \begin{equation}\label{dva}
    \pi(1-Q(t)) = t,
\end{equation}
since $\pi(0)=0$ and $Q(t)=1-F(0,t)$. Furthermore, \eqref{raz} and \eqref{dva} give an important representation
\begin{equation}\label{F_och_Q}
1-F(s,t) = Q(\pi(s)+t).
\end{equation}
After differentiating both sides of \eqref{rq} we find
\begin{equation}\label{to}
 \rho'(q(x))=\frac{1}{q'(x)}   
\end{equation}
which together with \eqref{rg} implies \eqref{q'}. \end{proof}

\begin{lemma}\label{l2} If \eqref{alpha} holds with $\alpha=0$, then \eqref{beta} implies \eqref{beta1}. On the other hand, if $\Ex[\nu]=1$, then \eqref{beta1} implies \eqref{alpha} with $\alpha=0$ and \eqref{beta}.
\end{lemma}
\begin{proof} First notice that
\[1-f(s)=(1-s)\sum_{k=0}^\infty \Prob(\nu>k)s^k\]
and similarly, since $\Ex[\nu]=\sum_{k= 0}^{\infty}\Prob(\nu>k)=1$,
\begin{equation}\label{ff}
    \frac{f(s)-s}{1-s}=(1-s)\sum_{k=0}^\infty \left(\sum_{i>k}\Prob(\nu>i)\right)s^k.
\end{equation}
Therefore, condition \eqref{alpha} with $\alpha=0$ is equivalent to
\[\sum_{k=0}^\infty \left(\sum_{i>k}\Prob(\nu>i)\right)s^k=(1-s)^{-1}L \left({1\over1-s}\right).\]
According to \cite[Cor.\ 1.7.3]{BGT} the latter is equivalent to
\[\sum_{i>k}\Prob(\nu>i)\sim L (k),\ k\to\infty\]
or in the integral form
\[\int_{x}^\infty\Prob(\nu>y)dy\sim L (x),\ x\to\infty.\]

Now given \eqref{beta} we apply \cite[Thm 1.7.2b]{BGT} to see that
\[\int_{z}^\infty\Prob(\nu>e^y)e^ydy\sim z^{-\beta}L_1(z),\ z\to\infty\]
entails 
\[\Prob(\nu>e^z)e^z\sim \beta z^{-1-\beta}L_1(z),\ z\to\infty\]
which is \eqref{beta1}. To prove the assertion in the opposite direction one should apply \cite[Thm 1.5.11]{BGT}.
\end{proof}

\begin{lemma}\label{l3} Under conditions of Theorem \ref{th} as $x\to\infty$
\begin{eqnarray}
 g'(x)&\sim& -\frac{\beta g(x)}{x},\label{g'}\\
  q''(x)&\sim& \frac{\beta q'(x)^2}{q(x)},\label{le2}\\
  \rho''(q(x))&\sim& \frac{\beta}{q'(x)q(x)}. \label{le3}
\end{eqnarray}
If we define $c(t)$ by
\begin{equation}\label{cc}
    {1\over c(t)}=q\left({1\over g(q(t))}\right),
\end{equation}
then it will satisfy \eqref{c}. 
\end{lemma}
\begin{proof} In the critical case we have
\begin{eqnarray*}
  1-f'(s) &=& (1-s)\sum_{k=0}^\infty \left(\sum_{i=k+2}^{\infty}i\Prob(\nu=i)\right)s^k \\
          &=& (1-s)\sum_{k=0}^\infty \Ex\left[\nu \mathbf{1}_{\{\nu\geq k+2\}}\right]s^k, 
\end{eqnarray*}
and on the other hand,  due to \eqref{ff}
\begin{eqnarray*}
  f(s)-s &=& (1-s)^2\sum_{k=0}^\infty \left(\sum_{i>k}\Prob(\nu>i)\right)s^k \\
          &=& (1-s)^2\sum_{k=0}^\infty \Ex\left[(\nu-k-1) \mathbf{1}_{\{\nu\geq k+2\}}\right]s^k, 
\end{eqnarray*}
These two relations together with $L((1-s)^{-1})=(f(s)-s)/(1-s)$ yield
\begin{align*}
  \frac{1}{(1-s)^2}L'\left(\frac{1}{1-s}\right)&={f'(s)-1\over1-s}+{f(s)-s\over(1-s)^2}\\
          &= -\sum_{k=0}^\infty \left(\Ex\left[\nu \mathbf{1}_{\{\nu\geq k+2\}}\right]-\Ex\left[(\nu-k-1) \mathbf{1}_{\{\nu\geq k+2\}}\right]\right)s^k\\
          &= -\sum_{k=0}^\infty (k+1)\Prob(\nu >k+1)s^k, 
\end{align*}
thus due to \eqref{beta1} as $s\to1$ 
\begin{eqnarray*}
  L'\left(\frac{1}{1-s}\right)&\sim&-\beta(1-s)^2\sum_{k=0}^\infty (\ln k)^{-1-\beta}L_1(\ln k)s^k. 
\end{eqnarray*}
Since due to \cite[Prop.\ 1.5.8]{BGT}
\begin{eqnarray*}
  \sum_{k=0}^n (\ln k)^{-1-\beta}L_1(\ln k)\sim n(\ln n)^{-1-\beta}L_1(\ln n),
\end{eqnarray*}
we derive from the previous relation applying \cite[Cor.\ 1.7.3]{BGT} 
\begin{eqnarray*}
  {1\over1-s}L'\left({1\over1-s}\right)&\sim&-\beta|\ln (1-s)|^{-1-\beta}L_1(-\ln (1-s)). 
\end{eqnarray*}
Now \eqref{g'} follows from the relations $g'(x)=e^xL'(e^x)$ and \eqref{beta}.

To derive \eqref{le2} it is enough to observe that $q''(x)=g'(q(x))q'(x)$ and use \eqref{g'}. From \eqref{to} we get $\rho''(q(x))=-q''(x)/q'(x)^3$. This and \eqref{le2} give \eqref{le3}. The last assertion of the lemma is a simple consequence of the basic properties of regular varying functions.
\end{proof}

\section{Proof of Theorems \ref{th} and \ref{the}}\label{s4}

The asymptotics \eqref{Q} of $Q(t)$ stated in Theorem \ref{th} follows from \eqref{r} and \eqref{rq} in view of \cite[Thm 1.5.12]{BGT}. 

We prove \eqref{lt} and Theorem \ref{the} after deriving an expression for the generating function of the reduced process
$$
\tilde{F}(s;u,t) = \Ex[s^{Z(u,t)}]
$$
in terms of $F(s,t)$. The survival probability at time $t$ for a branching process starting from a single particle at time $u$ is equal to $Q(t-u)$. Since the total number of particles alive at time $u$ is described by $F(s,u)$ we can write
$$\tilde{F}(s;u,t) = F(1-Q(t-u)+Q(t-u)s,u).$$ 
Combining this with the obvious relation
$$\tilde{F}(s;u,t)=1-Q(t)+\Ex[s^{Z(u,t)}|Z(t)>0]Q(t)$$
we deduce
\begin{align*}
\Ex[s^{Z(u,t)}|Z(t)>0] &= \frac{\tilde{F}(s;u,t)-1+Q(t)}{Q(t)} \notag\\
&= 1- \frac{1-F(1-Q(t-u)+Q(t-u)s,u)}{Q(t)}.
\end{align*}
 In view of \eqref{F_och_Q} it follows that
 \begin{align}
\Ex[s^{Z(u,t)}|Z(t)>0] &= 1- \frac{Q(\pi(1-Q(t-u)(1-s))+u)}{Q(t)}\notag\\
&=1-\exp\left\{q(t)-q(t+\Delta(s,t-u))\right\},\label{Fq}
\end{align}
where 
$$\Delta(s,t)=\rho(q(t)-\ln(1-s))-t.$$ 

Using \eqref{Fq} with $u=t$ we get
\begin{equation}\label{Fred_Rtt}
-\ln(1-\Ex[s^{Z(t)}|Z(t)>0]) = q(\rho(-\ln(1-s))+t)-q(t).
\end{equation}
To prove \eqref{lt} we study \eqref{Fred_Rtt} with $s=1-e^{-x/c(t)}$, where $x$ is a fixed positive number and $c(t)$ is defined by \eqref{cc}.
According to Lemma \ref{l3} and \eqref{r}  
$$\rho(x/c(t))\sim x^{1+\beta}/q'(t),\ t\to\infty,$$
and therefore, by a Taylor expansion around $t$,
\begin{align*}
-\ln(1-&\Ex[(1-e^{-x/c(t)})^{Z(t)}|Z(t)>0]) \\
&= q(\rho(x/c(t))+t)-q(t) \\
&= \rho(x/c(t))q'(t)+\rho(x/c(t))^2O(q''(t)) \\
&\to x^{1+\beta}.
\end{align*} 
Thus
$$
\Ex[(1-e^{-x/c(t)})^{Z(t)}|Z(t)>0] \to 1-e^{-x^{1+\beta}}
$$
and by the arguments of Darling \cite{Darling:1970} and Seneta \cite{Seneta:1973}, or Nagaev and Wachtel \cite{NW}, this implies that
$$
\Prob(\ln Z(tx,t) \leq xc(t)|Z(t)>0)\to 1-e^{-x^{1+\beta}},
$$
which finishes our proof of Theorem \ref{th}.

The proposed limit process $R(x)=Z_0(-\beta(1+\beta)^{-1}\log(1-x))$ in Theorem \ref{the} is a time inhomogeneous Markov branching process, and from \eqref{limit_dist} we find that
$$
\Ex[s^{R(y)}|R(x)=1] = 1 - (1-s)^{\left(\frac{1-y}{1-x}\right)^{\beta/(1+\beta)}},\, 0\leq x \leq y < 1.
$$
This yields
\begin{equation}\label{FSS-koppling}
\Prob(R(y)=1|R(x)=1)=\left(\frac{1-y}{1-x}\right)^{\beta/(1+\beta)}.
\end{equation}
By similar arguments that led to \eqref{Fq} we have 
$$
\Ex[s^{Z(v,t)}|Z(u,t)=1, Z_t>0] = 1-\exp\{q(t-u)-q(\Delta(s,t-v)+t-u)\},
$$
for $0\leq u \leq v < t$. In order to prove convergence of finite dimensional distributions it suffices to show
$$
q(\Delta(s,xt)+yt)-q(yt)\to -\left[\log(1-s)^{(x/y)^{\beta/(1+\beta)}}\right],\, 0< x\leq y \leq 1.
$$
We do this in two steps. First we find the asymptotics of $\Delta(s,t)$ by Taylor expansion of the function $\rho$ around $q$
\begin{eqnarray*}
 \Delta(s,t) &=&\rho(q(t)-\ln(1-s))- \rho(q(t))\\
   &=& -\ln(1-s)\rho'(q(t))+O(\rho''(q(t))) \\
   &\sim& -\ln(1-s)/q'(t).
\end{eqnarray*} 
And second we do another Taylor expansion, this time of the function $q$ around $yt$
\begin{align*}
q(yt+\Delta(s,xt))-q(yt)&= -\ln(1-s)\frac{q'(yt)}{q'(xt)}+(q'(t))^{-2}O(q''(t)) \\
&\to -\ln (1-s)\left(\frac{x}{y}\right)^{\beta/(1+\beta)},
\end{align*}
according to Lemma \ref{l3}. This finishes the proof of convergence of finite dimensional distributions in Theorem \ref{the}.

To show the convergence in the Skorohod sense, we can now copy the proof in \cite{FSS} almost verbatim. The only difference  is that we have \eqref{FSS-koppling} with $0<\beta/(1+\beta)<1$, whereas they in our notation have the expression $(1-y)/(1-x)$ for the corresponding probability. The difference in the exponent has no consequence for the proof.

\textbf{Acknowledgement.} We thank V.\ Wachtel for stimulating discussions of his recent paper \cite{NW}.

\end{document}